\documentclass[12pt,reqno]{amsart}

\usepackage{amsfonts,amssymb}
\usepackage{times,amsmath,amstext,amsbsy,amsopn,amsthm,upref,amscd,eucal}
\usepackage[all]{xy}
\usepackage{enumerate}

\allowdisplaybreaks

\newtheorem{theorem}{Theorem}[section]
\newtheorem{Luna Slice Theorem}[theorem]{Luna Slice Theorem}

\newtheorem{corollary}[theorem]{Corollary}
\newtheorem{lemma}[theorem]{Lemma}
\newtheorem{proposition}[theorem]{Proposition}
\newtheorem{proposition-definition}[theorem]{Proposition-Definition}

\theoremstyle{definition}
\newtheorem{definition}[theorem]{Definition}

\theoremstyle{remark}
\newtheorem{remark}[theorem]{Remark}

\newcommand{\rk}{\operatorname{rk}\nolimits}

\newcommand{\Card} {\operatorname{Card}\nolimits}

\newcommand\CC{{\mathbb C}}

\newcommand\QQ{{\mathbb Q}}

\newcommand\PP{{\mathbb P}}
\newcommand\ZZ{{\mathbb Z}}
\newcommand\GG{{\mathbb G}}
\newcommand\HH{{\mathbb H}}

\newcommand{\FFF}{{\mathcal F}}

\newcommand{\MMM}{{\mathcal M}}

\newcommand{\OOO}{{\mathcal O}}
\newcommand{\EEE}{{\mathcal E}}

\newcommand{\HHH}{{\mathcal H}}
\renewcommand\phi{\varphi}

\newcommand\phiti{\tilde{\phi}}

\newcommand{\im}{\operatorname{im}\nolimits}

\newcommand{\Pic}{\operatorname{Pic}\nolimits}
\newcommand{\pr}{\operatorname{pr}\nolimits}

\newcommand{\Sym}{\operatorname{Sym}\nolimits}

\newcommand{\tor}{\operatorname{tor}\nolimits}

\newcommand{\DGal}{\operatorname{DGal}\nolimits}

\newcommand\ra{{\rightarrow}}
\newcommand\rar{{\rightarrow}}

\begin{document}
%\title{KURANISHI Spaces}
\title{An illustration of a Shioda-Inose structure}
\author{Francois-Xavier Machu}
%\affil \endaffil
\address{F-X. M.: Mathematical Sciences Center, Tsinghua University
Room 131, Jin Chun Yuan West Building, Tsinghua University, Haidian District, Beijing 100084, China}
% \curaddr \endcuraddr
\email{xavier@math.tsinghua.edu.cn}
\thanks{Partially supported by start-up of Tsinghua University}

\bigskip
\bigskip
\begin{abstract}
We investigate the Shioda-Inose structure of the Jacobian of a smooth complex genus-$2$ curve $C$
arising from its degree-$2$ elliptic subcovers and determine the Mordell-Weil groups and lattices in the case
of a semistable fibration having exactly four singular fibers.

\end{abstract}
\maketitle
\vspace{1 ex}
\begin{center}\begin{minipage}{110mm}\footnotesize{\bf Key words:} Shioda-Inose structure, involution, lattice.
\end{minipage}
\end{center}

\begin{center}\begin{minipage}{110mm}\footnotesize{\bf MSC2000:} 14B12, 14F05, 14F40, 14H60, 32G08.

\end{minipage}
\end{center}
\vspace{1 ex}

\section*{} 
Let $C$ be a Weierstrass cubic which is one-dimensional integral scheme over $\CC$ of
arithmetic genus $1$ together with a point $p_0$ in its smooth locus. More precisely, it means that $C$ is either
an elliptic curve or a rational curve either with a node or a cusp. In other words, these two rational curves are degenerated
Kodaira fibers of types $I$ and $II$ (cycle of type $I_1$).
$C$ is embedded by linear system $\vert3p_0\vert$ as a cubic plane curve in the complex projective
plane satisfying the following equation
$$y^2z=x^3+g_2xz^2+g_3z^3,$$ where $g_2$ and $g_3$ are specified up to the $\CC^*$-action defined by 
$\mu.(g_2,g_3)=(\mu^4g_2,\mu^6g_3)$. We will refer to such an equation as a Weierstrass model.
The type of curves are actually of considerable importance in the
advancement to understand the concept of constructing moduli spaces
of holomorphic principal $G$-bundles over singular curves $C$ for 
$G$ a complex reductive algebraic group.
These curves appear in elliptic fibrations and make part
of fibers called Goreinstein arithmetic genus-one curves.
These curves and curves of types $I_N$ are classified essentially in both cases
where $f:X\ra S$ is an elliptic fibration with $X$ a fibered elliptic surface and
$S$ a smooth curve; and $f:X\ra S$ a morphism of smooth projective varieties where
$X$ is an elliptic three fold, having a section. This shows the importance
that play the elliptic surfaces in the classification of algebraic surfaces to determine the geometric and arithmetic properties
at the levels of groups, lattices and modular curves.

Let $X$ be an algebraic $K3$ surface defined over the field of complex numbers $\CC$.
Denote by $NS(X)$ the Neron-Severi lattice of $X$ is a sublattice of the unimodular lattice $H^2(X,\ZZ)$ of rank $(3,19)$.
This sublattice is an even lattice of signature $(1,\rho(X)-1)$, where $\rho$ denotes the rank of the Neron-Severi lattice of $X$.
The main property of the Nikulin involution $\phi$ for a $K3$ surface $X$ is that we can recover a $K3$ surface
$Y$ as the minimal resolution of the quotient of $X$ by the subgroup generated by $\phi$ obtained
from the blowing-up $\tilde{X}$ of $X$ at the locus of fixed points of $\phi$, which induces that $\phi$ extends to
$\tilde{\phi}$ of $\tilde{X}$ whose quotient by $<\tilde{\phi}>$ is smooth. 
Moreover, if we add a cocycle condition on the transcendental lattices of $K3$ surfaces and the $K3$ surface obtained $Y$ is a Kummer surface, then $X$ admits a Shioda-Inose structure.\\

Let $X$ be a $N$-polarized algebraic $K3$ surface defined over $\CC$, where
$N=U\oplus E_7\oplus E_8$ with $i:N\hookrightarrow NS(X)$ a lattice embedding whose image contains
a pseudo- ample class.
By Torelli Theorem with a Hodge structure
of weight $2$ on $T\otimes\QQ$, where $T$ is a rank-$5$ lattice
$U\oplus U\oplus (-2)$, this determine a bijective map:
$$(X,i)\leftrightarrow (A,\Pi),$$ where $A$ is a principally polarized abelian surface with its polarization $\Pi:A\ra J(A),$ where
$J(A)$ denotes the Jacobian of $A$.
On the one hand, the set of isomorphism classes of $N$-polarized
$K3$ surfaces with a canonical extension by the lattice $U\oplus E_8\oplus E_8$ is identified
with the set of isomorphism classes of Humbert surfaces $\HHH_1$, which identifies with the set of isomorphism classes 
of complex abelian surfaces in the form:
$$(E_1\times E_2,\OOO_{E_1\times E_2}((E_1\times\{p_2\}+\{p_1\}\times E_2))).$$
On the other hand, the set of isomorphism classes of $N$-polarized
$K3$ surfaces without extension is identified
with the open $\FFF_2\setminus{\HHH_1}$, where $\FFF_2=Sp(4,\ZZ)/\HH_2$ is the Siegel threefold, which identifies with the set of isomorphism classes of complex abelian surfaces
in the form:
$$(J(C),\OOO_{JC}(\theta)),$$ where $C$ is a smooth genus-$2$ curve.
Moreover, the Hodge correspondence provides an isomorphism of analytic spaces between 
the moduli space $\MMM_2$ of smooth genus-$2$ curves and $\FFF_2\setminus{\HHH_1}$.
Hence, in this case, we are able to determine the Shioda-Inose structure of the Jacobian of a smooth complex genus-$2$ curve $C$.
A number of properties for genus $1$-curves translate in the same way for $N$-polarized $K3$ surfaces or principally
abelian varieties obtained as minimal resolutions of a surface in $\PP^3$. This is a key to determine the Mordell-Weil groups and lattices of the semistable fibration over $\PP^1(\CC)$ having exactly four singular fibers.

\begin{theorem}
The Mordell-Weil groups of the semistable fibration over $\PP^1(\CC)$ having
exactly four singular fibers are:
\bigskip

\centerline{{\em
\begin{tabular}{|cr|cr|cr|}
\hline  $MW$ & $G(F_v)$ & Number of irreducible components of singular fibers & \\
 $(\ZZ/3)^2$ & $G(I_3)^2$ & $3,3,3,3$ &\\
%\hhline{|-|~|-|-|-|-|} 
$\ZZ/4\times\ZZ/2$ & $G(I_4)\times G(I_2)$ & $4,4,2,2$ &  \\
$\ZZ/5$ & $G(I_5)$ & $5,5,1,1$ &  \\
$\ZZ/6$ & $G(I_6)$ & $6,3,2,1$ & \\
$\ZZ/4$ &  $G(I_4)$ & $8,2,1,1$ & \\
$\ZZ/3$ & $G(I_3)$ & $9,1,1,1$ & \\
\hline
\end{tabular}
}}
\end{theorem} 

\begin{theorem}
We deduce the Mordell-Weil lattices of Beauville'semistable fibration.

\centerline{{\em
\begin{tabular}{|cr|cr}
\hline  $MWL$  & Number of irreducible components of singular fibers & \\
 $(\ZZ/3)^2/\ZZ/3$  & $3,3,3,3$ &\\
%\hhline{|-|~|-|-|-|-|} 
$(\ZZ/4\times\ZZ/2)/\ZZ/2$ &  $4,4,2,2$ &  \\
$\{0\}$  & $5,5,1,1$ &  \\
$\ZZ/6/\ZZ/3$ &  $6,3,2,1$ & \\
$\ZZ/4/\ZZ/2$  & $8,2,1,1$ & \\
$\{0\}$ &  $9,1,1,1$ & \\
\hline
\end{tabular}
}}\end{theorem}

The structure of the paper is as follows.

In Sect. $1$, we give an explicit description of the double covers of an elliptic curve.\\
In Sect. $2$, we use the Van Geemen-Sarti involution to investigate the Shioda-Inose structure arising from elliptic fibrations.\\
In Sect. $3$, we define the Mordell-Weil group and lattice of an elliptic surface.\\ 
In Sect. $4$, we illustrate these concepts to determine the Mordell-Weil groups and lattices for a semistable fibration over 
$\PP^1(\CC)$ having exactly $4$ singular fibers whose we extract a Van Geemen-Sarti involution.

\subsection*{Acknowledgements}
I am very grateful to Charles Doran and in particular to Jan Minac for their interest and encouragement of these results.

\section{Genus-2 covers of an elliptic curve}
\label{g2-covers}

In this section, we first will describe the degree-2 covers of elliptic curves
which are curves of genus $2$. 

\begin{definition}
Let $\pi:C\rar E$ be a degree-2 map of curves.
If $E$ is elliptic, then we say that $C$ is bielliptic and that
$E$ is a degree-2 elliptic subcover of $C$.
\end{definition}

Legendre and Jacobi \cite{J} observed that any genus-2 bielliptic curve
has an equation of the form
\begin{equation}\label{Jacobi}
y^2=c_0x^6+c_1x^4+c_2x^2+c_3\qquad (c_i\in\CC)
\end{equation}
in appropriate affine coordinates $(x,y)$.
It immediately follows that any bielliptic curve $C$ has two elliptic
subcovers $\pi_i:C\rar E_i$, 
\begin{equation}\label{Jacobi-Ei}\begin{array}{ll}
E_1:\ \ y^2=c_0x_1^3+c_1x_1^2+c_2x_1+c_3,\ \ \pi_1:(x,y)\mapsto
(x_1=x^2,y),\ \ \mbox{and}\\
E_2:\ \ y_2^2=c_3x_2^3+c_2x_2^2+c_1x_2+c_0,\ \ \pi_2:(x,y)\mapsto
(x_2=1/x^2,y_2=y/x^3).
\end{array}\end{equation}

This description of bielliptic curves, though very simple, depends on
an excessive number of
parameters. To eliminate unnecessary parameters, we will
represent $E_i$ in the form
\begin{equation}\label{E1-E2}
E_i: y_i^2=x_i(x_i-1)(x_i-t_i), \ \ (t_i\in\CC\setminus
\{0,\ 1\},\ t_1\neq t_2).
\end{equation}
Note that any pair of elliptic curves $(E_1,E_2)$ admits
such a representation even if $E_1\simeq E_2$.

We will describe the reconstruction of $C$ starting
from $(E_1,E_2)$ following \cite{Di}. This procedure will allow us to
determine the periods of bielliptic curves $C$ in terms of the periods
of their elliptic subcovers $E_1,E_2$.

Let $\phi_i:E_i\rar \PP^1$ be the double cover map $(x_i,y_i)\mapsto x_i$
($i=1,2$). Recall that the fibered product $E_1\times_{\PP^1} E_2$ is the
set of pairs $(P_1,P_2)\in E_1\times E_2$ such that $\phi_1(P_1)=
\phi_2(P_2)$. It can be given by two equations with respect to three
affine coordinates $(x,y_1,y_2)$:
\begin{equation}\label{fib_prod}
\bar C:=E_1\times_{\PP^1} E_2: \left\{\begin{array}{ll}
y_1^2=x(x-1)(x-t_1) \\
y_2^2=x(x-1)(x-t_2)
\end{array}\right.
\end{equation}
It is easily verified that $\bar C$ has nodes over the common
branch points $0,1,\infty$ of $\phi_i$ and is nonsingular elsewhere.
For example, locally at $x=0$, we can choose 
$y_i$ as a local parameter on $E_i$, so that $x$ has a zero of order
two on $E_i$; equivalently, we can write
$x=f_i(y_i)y_i^2$ where $f_i$ is holomorphic and $f_i(0)\neq 0$.
Then eliminating $x$, we obtain that $\bar C$ is given locally  by
a single equation $f_1(y_1)y_1^2=f_2(y_2)y_2^2$. This is the
union of two smooth transversal branches $\sqrt{f_1(y_1)}y_1=\pm
\sqrt{f_2(y_2)}y_2$.

Associated to $\bar C$ is its normalization (or desingularization)
$C$ obtained by divorcing the two branches at each singular point.
Thus $C$ has two points over $x=0$, whilst the only point of $\bar C$
 over $x=0$ is the node, which we will denote by the same symbol $0$. 
We will also denote by $0_+$, $0_-$ the two points of $C$ over $0$.
Any of the functions
$y_1,y_2$ is a local parameter at $0_\pm$. In a similar way, we introduce
the points $1,\infty\in \bar C$ and 
$1_\pm$, $\infty_\pm\in C$.

\begin{proposition}\label{fp}
Given a genus-$2$ bielliptic curve $C$ with its two elliptic subcovers
$\pi_i:C\rar E_i$, one can choose affine coordinates for $E_i$
in such a way that $E_i$ are given by the equations (\ref{E1-E2}),
$C$ is the normalization of the nodal curve $\bar C:=E_1\times_{\PP^1} E_2$, and
$\pi_i=\pr_i\circ\nu$, where $\nu:C\rar\bar C$ denotes the normalization map
and $\pr_i$ the projection onto the $i$-th factor.
\end{proposition}

\begin{proof}
See \cite{Di}.
\end{proof}

It is curious to know, how the descriptions given by \eqref{Jacobi} and 
Proposition \ref{fp} are related to each other. The answer is given by
the following proposition.

\begin{proposition}\label{deg6eq}
Under the assumptions and in the notation of Proposition \ref{fp},
apply the following changes of coordinates in the equations of
the curves $E_i$: 
$$
(x_i,y_i)\rar (\tilde x_i, \tilde y_i),\ \ 
\tilde x_i=\frac{x_i-t_j}{x_i-t_i},\ \ 
\tilde y_i= \frac{y_i}{(x_i-t_i)^2}\sqrt{\frac{(t_j-t_i)^3}{t_i(1-t_i)}},
$$
where $j=3-i$, $i=1,2$, so that $\{i,j\}=\{1,2\}$. Then the equations of
$E_i$ acquire the form
\begin{equation}\label{eqEi}\begin{array}{ll}
E_1:\ \ \tilde y_1^2=\left(\tilde{x_1}-\dfrac{t_2}{t_1}\right)\
\left(\tilde{x_1}-\dfrac{1-t_2}{1-t_1}\right)(\tilde{x_1}-1),\\
E_2:\ \ \tilde y_2^2=\left(1-\dfrac{t_2}{t_1}\tilde{x_2}\right)
\left(1-\dfrac{1-t_2}{1-t_1}\tilde{x_2}\right)(1-\tilde{x_2}).
\end{array}\end{equation}
Further, $C$ can be given by the equation
\begin{equation}\label{eqC}
\eta^2=\left(\xi^2-\frac{t_2}{t_1}\right)
\left(\xi^2-\frac{1-t_2}{1-t_1}\right)(\xi^2-1),
\end{equation}
and the maps $\pi_i:C\rar E_i$ by $(\xi,\eta)\mapsto 
(\tilde x_i, \tilde y_i)$, where
$$
(\tilde x_1, \tilde y_1)=(\xi^2,\eta),\qquad
(\tilde x_2, \tilde y_2)=(1/\xi^2,\eta/\xi^3).
$$
\end{proposition}

\begin{proof}
We have the following commutative diagram of double cover maps
\begin{equation}
\label{over-fields:diagram}
\xymatrix{
 &\ar_{\pi_1}[dl] C \ar_{f}[d] \ar^{\pi_2}[dr]\\
E_1 \ar_{\phi_1}[dr] & \ar_{\tilde{\phi}}[d] {\PP^1} & E_2\ar^{\phi_2}[dl] \\
& {\PP^1}
}
\end{equation}
in which the branch loci of $\phiti$, $\phi_i$, $f$, $\pi_i$
are respectively $\{t_1,t_2\}$, $\{0,1,t_i,\infty\}$, $\phiti^{-1}
(\{0,1,\infty\})$, $\phi_i^{-1}(t_j)$ ($j=3-i$). 
Thus the $\PP^1$ in the middle of the diagram can be viewed as the Riemann
surface of the function $\sqrt{\frac{x-t_2}{x-t_1}}$, where $x$ is the coordinate
on the bottom $\PP^1$. 
We introduce
a coordinate $\xi$ on the middle $\PP^1$ in such a
way that $\phiti$ is given by $\xi\mapsto x$, $\xi^{2}=\frac{x-t_2}{x-t_1}$.
Then $C$ is the double cover of $\PP^1$ branched in the 6 points
$\phiti^{-1}
(\{0,1,\infty\})=\{\pm1,\pm\sqrt{\frac{1-t_2}{1-t_1}},\pm\sqrt{\frac{t_2}{t_1}}\}$,
which implies the equation \eqref{eqC} for $C$. Then we deduce the
equations of $E_i$ in the form \eqref{eqEi} following the recipe of \eqref{Jacobi-Ei}, and it is an easy exercise to transform them into
\eqref{E1-E2}.
\end{proof}

From now on, we will stick to a representation of $C$ in the 
classical form $y^2=F_6(\xi)$, where $F_6$ is a degree-6 polynomial.
We want that $E$ is given the Legendre equation
$
y^2=x(x-1)(x-t),
$
but $F_6$ is not so bulky as in \eqref{eqC}. Of course, this can be done
in many different ways. We will fix for $C$ and $f$ the following choices:

\begin{eqnarray}\label{newsetting}%
f:C=\{y^2=(t'-{\xi^2})(t'-1-{\xi^2})(t'-t-{\xi^2)}\} & \ra &  
E=\{y^2=x(x-1)(x-t)
)\}\nonumber \\ 
(\xi,y) &  \mapsto& (x,y)=(t'-{\xi^2}, y)
\end{eqnarray}

\begin{lemma}
For any bielliptic curve $C$ with an elliptic subcover $f:C\rar E$ of degree
$2$, there exist affine coordinates $\xi,x,y$ on $C,E$ such that $f,C,E$ are given
by (\ref{newsetting}) for some $t,t'\in \CC\setminus \{0, 1\}$,
$t\neq t'$.
\end{lemma}

\begin{proof}
By Proposition \ref{fp}, it suffices to verify that the two
elliptic subcovers $E,E'$ of the curves $C$ given by \eqref{newsetting},
as we vary $t,t'$, run over the whole moduli space of elliptic curves
independently from each other. $E'$ can be determined from \eqref{Jacobi-Ei}.
It is a double cover of $\PP^1$ ramified at $\frac{1}{t'},
\frac{1}{t'-1},\frac{1}{t'-t},\infty$. This quadruple can be sent
by a homographic transformation to $0,1,t,t'$, hence $E'$ is given
by $y^2=x(x-1)(x-t)(x-t')$. If we fix $t$ and let vary $t'$, we will
obviously obtain all the elliptic curves, which ends the proof.
\end{proof}

The only branch points of $f$ in $E$ are 
$p_\pm =(t', \pm y_0)$, where $y_0=\sqrt{t'(t'-1)(t'-t)}$,
and thus the ramification points of $f$ in $C$ are 
$\tilde{p}_{\pm} =(0, \pm y_0)$. In particular, $f$
is non-ramified at infinity and the preimage of $\infty\in E$ is a
pair of points $\infty_\pm\in C$.\ \

We remark that if $\phi:C\ra E$ is a cover of degree $n$; then
there are two points $P_1,P_2$ in $C$ such that their multiplicity $e_{\phi}(P_i)=2, i=1..2$
and $e_{\phi}(Q)=1,\forall Q\in C\setminus\{P_1,P_2\}$ such that their images
are either different or the same. Otherwise, there is a point $P\in C$ with $e_{\phi}(P)=3$
and $e_{\phi}(Q)=1,\forall Q\in C\setminus\{P\}$. Thus $\phi$ is branched
in one point in two cases and two points in the other case. 

Let $\iota:C\ra C$ be the hyperelliptic involution map of $C$ whose fixed points are Weierstrass points $W$ contained
in the two-torsions points of $E$. We can specify the latter according to the 
degree of the cover. If its degree is odd, then $\phi(W)=E[2]$ and for any $Q\in E[2]$,
the cardinal of $\phi^{-1}(Q)\cap W\cong 1 (2)$. If $\deg\phi=2k$, then $\phi(W)\subset E[2]$
and and for any $Q\in E[2]$, the cardinal of $\phi^{-1}(Q)\cap W\cong 0 (2)$.

\section{Shioda-Inose structure}
\begin{definition}
A $K3$ surface $X$ admits a Shioda-Inose structure if there exists a Nikulin involution $\iota:X\ra X$
such for any holomorphic (2,0)-form $\omega$, then $\iota^{*}(\omega)=\omega$ and there is a rational quotient
map $\phi:X\dashrightarrow Y$ such that $Y$ is a Kummer surface with in addition a cocycle condition $T_X(2)\simeq T_Y$ which
expresses the isomorphism of transcendal lattices induced by the map $\phi_*$.
Moreover, this induces both degree-$2$ rational maps $\phi_1:X\dashrightarrow Y$ and $\phi_2:Z\dashrightarrow Y$, where
$Y$ is the Kummer surface associated to the 2-dimensional complex torus $Z$.
\begin{equation}
\label{over-fields:diagram}
\xymatrix{
X \ar_{\phi_1}[dr] &  & Z\ar^{\phi_2}[dl] \\
& {Y}
}
\end{equation}

\end{definition}

We give a criterion to determine if a $K3$ surface $X$ admits a Shioda-Inose structure.
\begin{theorem}
An algebraic $K3$ surface $X$ admits a Shioda-Inose structure if there exists a lattice primitive embedding
$T_X\hookrightarrow U^{\oplus 3}$.
\end{theorem}
\begin{proof}
See Theorem $6.3$ of \cite{Mor}.
\end{proof}

We come back our bielliptic curve $C$ arising from its elliptic subcovers $E$ and $E'$, and
investigate the Shioda-Inose structure.

\begin{theorem}
In an abstract way, we can construct a bielliptic curve $C$ from its elliptic
subcovers $E$ and $E'$ and conversely from the following commutative diagram
$$\xymatrix{ C\ar@{^{(}->}[r]^{\alpha} \ar@{^{(}->}[d]_{f\times f'} & JC   \\ E\times E'  \ar[ur]_{f^{*}+f'^{*}}} $$
in which $f^{*}+f'^{*}$ is an isogeny of degree $2$ and $f\times f'(C)$ is the graph of a $(2, 2)$ correspondence between  $E, E'$.
From this, we have
$$\xymatrix{SI(J(C))\ar[r]^{2:1} \ar[d]_{2:1} & Kum(J(C))   \\ Kum(J(E)\times J(E'))  \ar[ur]_{2:1}},$$
where $SI(JC)$ denotes the Jacobian of $C$ endowed with a Shioda-Inose structure corresponding to a $N$-polarized $K3$ surface
without extension.
\end{theorem} 
\begin{proof}
We can first assume that $C$ is a generic smooth genus-$2$ curve (i.e. such that its Jacobian surface
has $\rho(JC)=1$) and prove in this case that the morphism  $Kum(J(E)\times J(E'))\ra Kum (J(C))$ is an isomorphism.
$JC$ is a principally polarized abelian variety and if $H$ is the principal
polarization, then we have $H^2=2$ and $T=T_{JC}=U^{\oplus 2}\oplus<-2>.$  
Note that $T$ embedds in $\Lambda$, then there exists a $K3$ surface $X$ such that
we have a Hodge isometry $T\simeq T_X$. We construct an embedding lattice from
$T_X\hookrightarrow U^{\oplus 3}$ in the following way, we send the first two copies
of $U\subset T_X$ to the corresponding ones of $U^{\oplus 3}$ and the remaining 
element $-2$ to $e_1^3-e_2^3$, where $(e_i^{j})_{j=1\dots 3,i=1,2}$, is a basis of $U^{\oplus 3}$.
Then by the previous theorem, $X$ admits a Shioda-Inose structure and we have the following diagram,
\begin{equation}
\label{over-fields:diagram}
\xymatrix{
X \ar_{\phi_1}[dr] &  & Z=E\times E'\ar^{\phi_2}[dl] \\
& {Y}
}
\end{equation}
with Hodge isometries $T_X\simeq T_{JC}$ and $T_Z\simeq T_{JC}$.
Then these are Fourier-Mukai partners and $Kum(JC)\simeq Kum(Z)$ after Main Theorem of \cite{H-L-O-Y}.
We at present consider the framework general of a smooth genus-$2$ curves.
As $JC$ and $E\times E'$ are only two-isogeneous, this only implies an isomorphism of $T_{JC}$ and $T_{E\times E'}$ over $\QQ$ not over
$\ZZ$; and we do not recover the isomorphism in the generic case.
The Hodge correspondence gives an isomorphism
between the set of isomorphism classes of $N$-polarized K3 surfaces without extension and this one of principally polarized abelian
varieties $(JC,\OOO(\Theta))$. This induces an isomorphism between the open region $\FFF_2\setminus\{\HHH_1\}$ and the moduli space of the genus-$2$ curves. Then we rely on the first parts of the Sect. $4$ of the paper of Clingher-Doran \cite{Cl-D1} to get the diagram given above. 
\end{proof}

For consequently, by lemma $4.1$ of \cite{Cl-D1}, the Shioda-Inose structure of the Jacobian of a smooth genus-$2$ curve is:
\begin{proposition}
The Shioda-Inose structure of the Jacobian of a smooth complex genus-$2$ curve is a
$K3$-surface $X(k,m,n)$ with transcendental Hodge isometric lattice $T(k,m,n)$, where
$T(k,m,n)=U(k)\oplus U(m)\oplus<-2n>.$ Moreover by Hodge correspondence, the $K3$-surface $X(k,m,n)$ is
represented by a principally polarized abelian surface of type $(1,n)$. 
\end{proposition}

\section{Mordell-Weil group and lattice of an elliptic surface}
We want to determine the Mordell-Weil group and lattice of an elliptic surface in a particular case.
So, we need to put in place some materials to realize this.
We first need to know what happens at the level of the fibers after a change base on the new constructed elliptic surface
and to solve its singularities.

We start with an elliptic surface $S$ over a smooth curve $C$. The change base of $S$ requires
a smooth projective curve $B$ together with an onto morphism $\pi:B\ra C$ such that
the following diagram commutes 
\begin{equation}
\label{over-fields:diagram}
\xymatrix{ S'\ar[r]^{\tilde{f}}\ar[d]_{\rho}&B\ar[d]^{\pi}\\ S \ar[r]^{f} & C},
\end{equation}
where $S'=S\times_C B$ is the constructed elliptic surface.
If $S\ra C$ the elliptic surface has only smooth fibers then those translate into $S'$ by the same
smooth fibers whose the number copies is the degree of $\pi$.
In the case where it only appears in the elliptic surface $S$ of multiplicative fibers $I_n$ over unramified
points of $\pi$, then those translate into $S'$ in $I_{nd}$, where $d$ is the index of ramification of $\pi$.
In the case where it only appears of additive fibers, we make reference
to the Tate's table (See \cite{Sc-Sh}). 

We at present make refer to T.Shioda \cite{Sh1} and \cite{Sh2} to define in a first time Mordell-Weil groups and in a 
second time the Mordell-Weil lattices of an elliptic surface.
Let $K=\CC(C)$ be the function field on a smooth complex projective curve $C$.
Let $E$ be an elliptic curve defined over $K$ with a $K$-rational point.
Let $E(K)$ be the group of $K$-rational points of $E$, with origin $O$.
The Kodaira-N\'eron model of $E/K$ is an elliptic surface $f:S\ra C$, where
$S$ is a smooth complex projective surface and $f$ has no exceptional $(-1)$ curves.
So, the group $E(K)$ of $K$-rational points of $E$ called the Mordell-Weil group of $E$ can be indentified with the group
of sections of $f$. Each $P\in E(K)$ determines a section of $f$ is interpreted as a divisor on $S$, which
is a curve denoted $\bar{P}$.

Assume throught that $f$ has at least one singular fiber. Then $E(K)$ is a finitely
generated (Mordell-Weil theorem) and the N\'eron-Severi group $NS(S)$ of $S$ is none
than $\Pic(S)$, which becomes an integral lattice of finite rank with respect to the
intersection pairing $D.D'$. Denote us $T$ its sublattice generated by the zero section $(O)$,
by a general fiber $F$ and the irreducible components of fibers. Then $T$ can be written as follows:
$$T=(\ZZ(O)\oplus\ZZ F)\oplus(\oplus_{v\in N}T_v),$$ where $N=\{v\in C\mid F_v \ \textrm{is reducible}\}$,
and $T_v$ is generated by the irreducible components of $F_v$ not meeting the zero section $(O)$.

The map $P\mapsto\bar{P}$ $\mod T$  induces a group isomorphism:
$$E(K)\simeq NS(S)/T,$$ and then a unique homomorphism:
$$\phi:E(K)\ra NS(S)\otimes\QQ$$ such that 
$$\phi(P)=\bar{P} \mod T\otimes\QQ, \ \im(\phi)\bot T.$$

Now, we want to consider the torsion sections and how those work in the case of an elliptic surface
with singularities.
Assume that $P\in E(K)$ has torsion, then its image $\bar{P}$ lies in the primitive closure $T'$ of $T$ defined
as follows $$T'=(T\otimes\QQ)\cap NS(S).$$ Denote $t_p$ the automorphism of elliptic
surfaces without forgetting their elliptic structure. Then the group $<t_p>$ is finite of order $m$ that acts
on the generic fiber by fixed point free and the quotient
of $S$ by $<t_p>$ gives an elliptic curve $E'$ together with an isometry from $E'\ra E$, whose
Kodaira-N\'eron model is an elliptic surface obtained as the minimal resolution of the
quotient $\frac{S}{<t_p>}$.
In the case where the characteristic of the field $p$ does not divide $m$, then we can
study the type of Kodaira fibers appear in the elliptic surface since the operation
is separable in the sense that the $j$-form of an elliptic surface is defined as the morphism $j:C\ra\PP^1(\CC)$
that is separable.
We treat the various cases separetly. In the multiplicative fibers, if 
we assume that $m$ is a prime number to simplify the case, and $P$ meets
the zero component of an $I_n$ fiber, then each component of $I_n$ is fixed by $<t_p>$.
Hence, there are  $n$ fixed points at the intersection of components. Each of them
attains an $A_{m-1}$ singularity in the quotient whose desingularization is $I_{mn}$.
If $P$ does not meet the zero section, then $m\mid n$ by corollary \ref{cr}. $t_p$ rotates the singular fiber
($i.e$ cycle of rational curves) by an angle of $\frac{2\pi}{m}$ and the resulting fiber in $S'$ has
type $I_{\frac{n}{m}}$. In the case of additive fibers, $P$ has to meet the non-trivial singular fiber.
\begin{corollary}\label{cr}
Restricted to the torsion subgroup of $E(K)$, the group homomorphism $\phi$ is one-to-one:
$$\phi:E(K)_{tors}\ra\prod_{v\in R} G(F_v).$$
\end{corollary}
\begin{proof}
See $11.9$ of \cite{Sc-Sh}.
\end{proof}

The case among which we are interested is the case of an elliptic rational extremal elliptic surface.

\begin{definition}
An elliptic surface $S$ with section is said extremal if $\rho(S)$ is maximal and $E(K)$
is finite. 
\end{definition}
\begin{remark}
$\rho(S)$ maximal means that the bounds given by Igusa and Lefschetz are attained.
In other words, $\rho(S)=h^{1,1}(S)$ (resp. $=b_2(S)$).
\end{remark}

An extremal rational elliptic surface is featured by the discriminant of its trivial lattice as
follows $$disc(T(S)=-(\#E(K))^2.$$
We use the following corollary to deduce the number of singular fibers in such an elliptic surface.
\begin{corollary}\label{cr1}
Let $S$ be an elliptic surface with section. Denote the generic fiber by $E$. Then
$$\rho(S)=\rk(T(S))+\rk E(K)=2+\sum_{v\in R}(m_v-1)+\rk(E(K)).$$
\end{corollary}
\begin{proof}
See section $6$ of \cite{Sc-Sh}.
\end{proof}
Hence, it will appear at least four singular fibers in this elliptic surface whose
classification is:\\
$(1)$ 4 multiplicative singular fibers.\\
$(2)$ 2 multiplicative and 1 additive singular fibers.\\
$(3)$ 2 additive singular fibers.\\
(We denote in the case where there is a wild ramification, we have at the same
time of additive and multiplicative singular fibers.)\
One classifies the extremal rational elliptic surfaces with section in terms 
of configuration of singular fibers which we write as tuples
$$[[n_1,n_2,...]], with \ entries\ 1,2,...,0^*,1^*,...,II,III,...,II^*,$$
representing the fiber types. 
This configuration is a method of determining the Mordell-Weil group.
One can use quotients by translation by torsion sections to limit
the possible orders. Denote that is no longer true in the case of K3 elliptic surfaces.

In our case, we are only interested in the multiplicative fibers. Then,the existence of the
Mordell Weil groups is ensured by corollary \ref{cr}. 

We want to compute the euler number $e(S)$ of an elliptic surface $S$.
It is known that in the case of fiber $F_v$, we get the following results.
\begin{equation}\label{fib_prod}
e(F_v)=\left\{\begin{array}{lll}
0, if \ F_v \ is\ smooth;\\
m_v, if \ F_v \ is \ multiplicative;\\
m_v+1, if \ F_v \ is \ additive.\\

\end{array}\right.
\end{equation}

\begin{theorem}
For an elliptic surface $S$ over $C$, we have
$$e(S)=\sum_{v\in C} (e(F_v)+\delta_v),$$
where $\delta_v$ is the wild ramification defined as
$\delta_v=v(\Delta)-1$-numbers of components of an additive fiber.
\end{theorem}
\begin{proof}
See Proposition $5.16$ of \cite{Co-D}.
\end{proof}

We at present define the Mordell-Weil lattice of an elliptic surface.

For $P,P'\in E(K)$, let $$<P,P'>=-\phi(P)\phi(P').$$
Then with this height pairing the group $E(K)/E(K)_{\tor}$ becomes
a positive definite lattice, called the Mordell-Weil lattice of $E/K$ or of
$f:S\ra C$. These are not integral lattices in general.

We can rewrite down the pairing defined above as follows:
$$<P,P'>=\chi+P.O+P'.O-P.P'-\sum_{v\in N}\textrm{contrv(P,P')}.$$
Here $\chi$ is the euler charasteristic of the surface $S$, $P.Q=\bar{P}.\bar{Q}$,
and the term in the sun expresses that $P$ and $P'$ pass through
non-identity components of $F_v$.

The subgroup $E(K)^0$ of $E(K)$ consisting of those sections
meeting the identity of every fiber is a torsion-free subgroup
of $E(K)$ of index finite, and becomes a positive definite even
lattice, called the narrow Mordell-Weil lattice of $E/K$ or of $f/S\ra C$.
This lattice is isomorphic via the map $\phi$ to the opposite of the
orthogonal complement $T^{\bot}$ of $T$ in the N\'eron-Severi lattice $NS(S)$.

Note that $$<P,P'>=\chi+P.O+P'.O-P.P'$$ if $P$ or $P'$ are in $E(K)^0$,
$$<P,P>=2\chi+2P.O\geq 2\chi,$$ for $P\in E(K)^0, P\neq O$.

\section{Semistable fibration-Beauville's case}
We are going to illustrate these various concepts in a particular case
where the fibration is semistable said of Beauville.

\begin{definition}
Let S be a surface and $C$ a smooth projective curve, $f:S\ra C$ is said a semistable fibration
if $S$ is a smooth surface and the fibers are connected rational curves of genus $g\geq 1$ having
at worse ordinary double points. Moreover, the fibers do not contain the exceptional curves.
\end{definition}

From (\ref{eqEi}), we can construct the following diagram of semistable fibrations.
\begin{proposition}
\begin{equation}
\label{over-fields:diagram}
\xymatrix{
 &\ar_{\pi_1}[dl] X \ar_{f}[d] \ar^{\pi_2}[dr]\\
X' \ar_{\phi_1}[dr] & \ar_{\tilde{\phi}}[d] {\tilde{X}} & X''\ar^{\phi_2}[dl] \\
& {\PP^1}
}
\end{equation}
where $\tilde{\phi}$ is a semistable fibration having $4$ singular fibers
of genus $1$, $\phi_1, \phi_2$ are semistable fibrations consisting of
$6$ singular fibers of genus $3$, $f$ from a $4$-fold $X$ is a double cover of $\tilde{X}$ algebraic surface and the $\pi_i, i=1..2$ are 
double covers of semistable fibrations $X'$ and $X''$ are both algebraic surfaces. The composition $f\circ\tilde{\phi}$ is
a semistable fibration with $8$ singular fibers of genus $5$.    
\end{proposition}

\begin{proof}
The proof is based on the arguments given by Beauville in \cite{Be-1}. So, for the reader,
we remind this proof.
We start with a smooth curve $C$, endowed with a morphism $\phi:C\ra\PP^1(\CC)$ of degree $n$ and
a homography $u$ of $\PP^1(\CC)$ such that the following conditions hold:\\
$(i)$ the set $R$ of 
the ramification points are simple; \\
$(ii)$ $R$ is stable under $u$ and has no fixed points of $u$.
So, there exists a double cover $\pi:X'\ra C\times\PP^1(\CC)$ ramified
along $\Gamma_{\phi}\cup\Gamma_{u\circ\phi}$ which are linearly equivalent divisors since
$PGL(2,\CC)$ is a rational variety. Let $g=\pr_2\circ u$ be the composition morphism.
Let $t\in\PP^1(\CC)$ be a point, the fiber $g^{-1}(t)$ is also a double cover ramified
along the divisor $\phi^{-1}(t)+\phi^{-1}(u^{-1})(t)$. For either $t\neq 0$ or a fixed
point of $u$, the divisor is not reduced and multiplicity $2$. By blowing-up the
double points of $X'$, we get a semistable fibration $X\ra\PP^1(\CC)$, that admits
$\Card(R)+2$ singular fibers whose genus is $n-1+2g(C)$.
On the one hand, we can construct a morphism of degree $n$ even having four simple
ramification points $\phi:E\ra\PP^1(\CC)$, $g(E)=1$, obtained as the composition of a double cover 
ramified into these $4$ points with an \'etale morphism of degree $\frac{n}{2}$.
Hence, in our case, for $n$=2, $g^{-1}(t)$ is a singular genus-$2$ curve and 
get a semistable fibration having $6$ singular fibers whose genus is $3$. 
We proceed in the same way of constructing another algebraic surface $X''$ and deduce the $4$-fold $X$.
On the other hand, as $C$ is a smooth genus-$2$ curve and $f:C\ra\PP^1(\CC)$ is branched
in $6$ points, so composition $f\circ\tilde{\phi}$ is a semistable fibration with $8$ singular fibers of genus $5$.
\end{proof}

Note if $n$ is odd, we construct a morphism $\phi:\PP^1(\CC)\ra\PP^1(\CC)$ of degree $n$
having four simple ramification points. Then, we get a semistable fibration
having six singular fibers whose genus is $n-1$.\\
We remark that for $n=3$, 
a semistable fibration at $6$ singular fibers of the genus-$2$ curves.

We want to determine the Mordell-Weil groups and lattices of the semistable fibration having four singular fibers 
over $\PP^1(\CC)$.

We first set up a few notation.
Let $\Gamma$ be the group of 
index finite in $SL_2(\ZZ)$ satisfying the relation $(SS)$: $\Gamma$ whose trace is different from
the values $\{-2,-1,0,1\}.$
Note that $\Gamma$ acts freely over the upper-half plane $\HH$ ; the semi-direct product
of $\Gamma$ by $\ZZ^2$ acts freely and properly on $\HH\times\CC$ by the formula:
$$(\gamma,p,q).(\tau,z)=(\gamma\tau,(c\tau+d)^{-1}(z+p+q)),$$ for
\begin{eqnarray}
\gamma=\left( \begin{array}{cc} a & b \\ c & d
\end{array}\right)\in\Gamma,
\end{eqnarray}
$$(p,q)\in\ZZ^2, \tau\in\HH, z\in\CC.$$

We denote $X_{\Gamma}^0$ the quotient surface, and $B_{\Gamma}^0$ the curve $\HH/\Gamma$. By $(2)$, the
smooth elliptic fibration extends in a only one way in a semistable fibration $X_{\Gamma}\ra B_{\Gamma}$:
this is the modular family associated to $\Gamma.$

We consider the following subgroups of $SL_2(\ZZ)$:
\begin{eqnarray}
\Gamma(n)=\{\left( \begin{array}{cc} a & b \\ c & d
\end{array}\right)\in SL_2(\ZZ)\mid b\equiv c\equiv 0, a\equiv 1 (mod.n)\}  ,
\end{eqnarray}

\begin{eqnarray}
\Gamma_0^{0}(n)=\{\left( \begin{array}{cc} a & b \\ c & d
\end{array}\right)\in SL_2(\ZZ) c\equiv 0, a\equiv 1 (mod.n)\}  ,
\end{eqnarray}

\begin{eqnarray}
\Gamma_0(n)=\{\left( \begin{array}{cc} a & b \\ c & d
\end{array}\right)\in SL_2(\ZZ)\mid b\equiv c\equiv 0 (mod.n)\}  ,
\end{eqnarray}

Consider elsewhere a pencil of cubics in $\PP^2$, such that the only one singularities
of the pencil are ordinary double points. By blowing-up of the nine base points of the pencil,
we have an elliptic semistable fibration over $\PP^1$, said deduced family of 
the pencil of cubics.

We state a theorem from Beauville (See \cite{Be-2}).
\begin{theorem}
Let $f:X\ra\PP^1(\CC)$ be an elliptic semistable fibration having exactly four singular fibers. Then
$f$ is isomorphic to the family of the modular curves associated to the group $\Gamma$ of 
index finite in $SL_2(\ZZ)$ satisfying the relation $(SS)$; which identifies with the
family of the cubics induced by a pencil of cubics:
\bigskip

\centerline{{\em
\begin{tabular}{|cr|cr|cr|}
\hline  $\Gamma$ & Pencil Equations & Number of irreducible components of singular fibers & \\
\hline $\Gamma(3)$ & $X^3+Y^3+Z^3+tXYZ=0$ & $3,3,3,3$ &\\
%\hhline{|-|~|-|-|-|-|} 
$\Gamma_0^{0}(4)\cap\Gamma(2)$ & $X(X^2+Z^2+2ZY)+tZ(X^2-Y^2)=0$ & $4,4,2,2$ &  \\
$\Gamma_0^{0}(5)$ & $X(X-Z)(Y-Z)+tZY(X-Y)=0$ & $5,5,1,1$ &  \\
$\Gamma_0^{0}(6)$ & $(X+Y)(Y+Z)(Z+X)+tXYZ=0$ & $6,3,2,1$ & \\
$\Gamma_0(8)$ &  $(X+Y)(XY-Z^2)+tXYZ=0$ & $8,2,1,1$ & \\
$\Gamma_0(9)\cap\Gamma_0^{0}(3)$ & $X^2Y+Y^2Z+Z^2X+tXYZ=0$ & $9,1,1,1$ & \\
\hline
\end{tabular}
}}\bigskip

\end{theorem}

We adapt this to the case of Beauville'semistable fibration, where
$K=\CC(\PP^1)$, the field of rational functions over the complex projective line, and
$T=(\ZZ(O)\oplus+\ZZ F)\oplus(\sum_{i=1}^{12}\oplus T_i)$. 
Note that Beauville'semistable fibration belongs to the family
of extremal rational elliptic surfaces. Then the Kodaira-N\'eron model of this
elliptic surface will define a structure of an algebraic group scheme over all the fibers.
Hence, the group scheme of the identity component is $\GG_m$ if $F_v$ is multiplicative and
$\GG_a$ if $F_v$ is additive. The quotient by this group scheme is a finite abelian group
will be denoted $G(F_v)$ depending on the fiber.
\begin{theorem}
The Mordell-Weil groups of the semistable fibration over $\PP^1(\CC)$ having
exactly four singular fibers are:
\bigskip

\centerline{{\em
\begin{tabular}{|cr|cr|cr|}
\hline  $MW$ & $G(F_v)$ & Number of irreducible components of singular fibers & \\
 $(\ZZ/3)^2$ & $G(I_3)^2$ & $3,3,3,3$ &\\
%\hhline{|-|~|-|-|-|-|} 
$\ZZ/4\times\ZZ/2$ & $G(I_4)\times G(I_2)$ & $4,4,2,2$ &  \\
$\ZZ/5$ & $G(I_5)$ & $5,5,1,1$ &  \\
$\ZZ/6$ & $G(I_6)$ & $6,3,2,1$ & \\
$\ZZ/4$ &  $G(I_4)$ & $8,2,1,1$ & \\
$\ZZ/3$ & $G(I_3)$ & $9,1,1,1$ & \\
\hline
\end{tabular}
}}
\end{theorem} 
\begin{proof}
Let $j\in K\subset\{0,12^3\}$. Then the following elliptic curve in generalized Weierstrass form has
$j$-invariant:
\begin{equation}\label{ell.c}
E: y^2+xy=x^3-\frac{36}{j-12^3}x-\frac{1}{j-12^3}.
\end{equation}

Denote that there are both another elliptic curves admitting extra automorphisms
defined by: $$j=0: y^2+y=x^3, \ j=12^3: y^2= x^3+x,$$ in characteristic of $K$ distinct
from $2$ and $3$.

We use the normal form for an elliptic surface with $j$-given.
The generic fiber is defined over $K(t)$. An integral model is obtained
from (\ref{ell.c}) by a simple scaling of $x$ and $y$ of the type $x\mapsto u^2x$ and $y\mapsto u^3y$.
$$S: y^2+(t-12^3)xy=x^3-36(t-12^3)^3x-(t-12^3)^5.$$

We want to determine the singular fibers of $S$. We first compute
the discriminant $\Delta=t^2(t-12^3)^9.$ As the characteristic of $K$ is distinct from $2$ and $3$, 
We carry out of translations in $x$ and $y$ to yield the Weierstrass form
$$y^2=x^3-\frac{1}{48}t(t-12^3)^3x+\frac{1}{864}t(t-12^3)^5.$$
Then, we deduce that the fibers at $0$ and $12^3$ are additive fibers which are of type $II$ (resp. of type $III^*$)
arising from the Tate's table.
Otherwise, the two additive fibers collapse so that the valuation of the discriminant $\Delta$ at $0$ becomes
$11.$ Tate'algorithm only terminates at fiber of type $II^*$.
As concerns the point at the infinity, we carry out the following change of coordinates
$$t\mapsto\frac{1}{s},\ x\mapsto\frac{x}{s^2},\ y\mapsto\frac{y}{s^3}.$$ 
Hence the discriminant $\Delta=s(1-12^3s)^9,$ since the vanishing order of $\Delta$ at 
$s=0$ is $1$, $S$ acquires a nodal rational curve.

After a suitable base change, the additive fibers are replaced by semistable fibers
in the case of semistable fibration of Beauville that are of multiplicative fibers of type $I_n$.

In our case, $j:C\ra\PP^1(\CC),$ is a morphism of degree 
$e(S)=12$. According to the Tate's table, this guarantees that
in the pull-back of the normal surface via $j$, the additive fibers
$II$ and $III^*$ are replaced by $4$ fibers (resp. $6$) fibers
of type $I_0^*$. Then these can be eliminated by quadrating twists
to recover the semistable fibration $S\ra C$. Hence, $S$ acquires
of multiplicative fibers $I_n$.

We remark that in the case where $K$ is an algebraically closed field, we do not need
to use quadratic twist to ensure the uniqueness of elliptic curves
up to isomorphisms.

We determine the Mordell-Weil group for the case $[1,1,1,9]$
of a Beauville' semistable fibration.
To carry out this, we transform the cubic equation after a suitable base change
into a Weierstrass form $E: y^2+a_1 xy+a_3 y=x^3+a_2x^2$, with $P=(0,0)$ is not
$2$-torsion. Since $-2P=(a_2,0)$ is $3$-torsion if only if $a_2=0$. As we work
in an extremal rational elliptic surface, so, after a suitable base change, we
can recover both minimal rational elliptic surfaces defined as follows.
$$S: y^2+xy+ty=x^3, \ and \ S': y^2+ty=x^3.$$
The trivial lattice has rank $10$ as $e(S)=e(S')=12$, and $\rho(S)=10$.
We deduce from corollary (\ref{cr1}) that the Mordell-Weil rank is zero.
Since $T$ has discriminant $-9$, there can only be $3$-torsion in either case.
$$E(K)=E'(K)=\{O,P,-P\}\simeq\ZZ/3.$$
We can also get this result from the height.
$E(K)$ is torsion free with generator $P$ meeting the component $\Theta_3$ of the
fiber $I_9$ whose $h(P)=2-3.5/9=1/3$, so that $P\in 1/3\ZZ,$ and we find
the Mordell-Weil group $\ZZ/3$ again. The another cases of determining the Mordell-Weil groups
treat in the similar way.
\end{proof}

Note that in the second case, it arises a Van-Geemen Sarti involution.

We at present want to determine the Mordell-Weil lattices of Beauville'semistable fibration.
Hence, we first establish the list of the torsion parts of the various Mordell-Weil groups in this case.

\begin{corollary}
The subgroups $E(K)^0$ of the Mordell-Weil of the semistable fibration over $\PP^1(\CC)$ having
exactly four singular fibers are:
\bigskip

\centerline{{\em
\begin{tabular}{|cr|cr|cr|}
\hline  $E(K)^0$ & $G(F_v)$ & Number of irreducible components of singular fibers & \\
 $\ZZ/3$ & $G(I_3)^2$ & $3,3,3,3$ &\\
%\hhline{|-|~|-|-|-|-|} 
$\ZZ/2$ & $G(I_4)\times G(I_2)$ & $4,4,2,2$ &  \\
$\ZZ/5$ & $G(I_5)$ & $5,5,1,1$ &  \\
$\ZZ/2,\ZZ/3$ & $G(I_6)$ & $6,3,2,1$ & \\
$\ZZ/2$ &  $G(I_4)$ & $8,2,1,1$ & \\
$\ZZ/3$ & $G(I_3)$ & $9,1,1,1$ & \\
\hline
\end{tabular}
}}
\end{corollary} 
\begin{proof}
The proof is based on the same arguments than the proof of
determining the Mordell-Weil groups.
\end{proof}

\begin{theorem}
Then, we deduce the Mordell-Weil lattices of Beauville'semistable fibration.

\centerline{{\em
\begin{tabular}{|cr|cr}
\hline  $MWL$  & Number of irreducible components of singular fibers & \\
 $(\ZZ/3)^2/\ZZ/3$  & $3,3,3,3$ &\\
%\hhline{|-|~|-|-|-|-|} 
$(\ZZ/4\times\ZZ/2)/\ZZ/2$ &  $4,4,2,2$ &  \\
$\{0\}$  & $5,5,1,1$ &  \\
$\ZZ/6/\ZZ/3$ &  $6,3,2,1$ & \\
$\ZZ/4/\ZZ/2$  & $8,2,1,1$ & \\
$\{0\}$ &  $9,1,1,1$ & \\
\hline
\end{tabular}
}}\end{theorem}

The invariants of Mordell-Weil lattices can be expressed in terms of the geometric data
of the surface $S$. Let $M=E(K)^0$ be the narrow Mordell-Weil lattice of $f:S\ra C$.
Then 
$$\rk(M)=\rho(S)-2-\sum_{v\in N}(m_v-1),$$ lies in $\{0,1\}$ for the Beauville's case.
$$\det(M)=\nu^2\vert\det(NS(S)/\det(T)\vert, \nu=[E(K):E(K)^0],$$ lies in $\{8,9,18,25,108,162\}$ for the Beauville's case.
$$\mu(M)=2\chi+2\min\{P.O:P\in M,P\neq O\}\geq 2\chi,$$
where $\rho(S)$ is the Picard number of $S$ and $m_v$ is the number
of irreducible components of the singular fibers $F_v$, and $\mu(M)$
is the square of the minimal norm of $M$.

Note that there exits another approach of determining the Mordell-Weil lattices of Beauville's semistable fibration.
We just give a sketch of this approach.
Consider the Lam\'e connections which are defined as the irreducible rank $2$-connections on 
$\PP^1(\CC)$ having four regular singular points at $0,1,t,\infty$ with
exponents $-1/2,1/2$ (see \cite{LVU} for more details .\\ Let $f:E\ra\PP^1$ be the double cover, its generic fiber
is a two dimensional irreducible differential module $M$ over $C(z)$, the
field of rational functions on $\PP^1(\CC)$. Then
$M$ is imprimitive and  $\Sym^2M$ has a one differential module such that
the corresponding quadratic extension $L$ has the form $L=C(z)(w)$.\\
We want to investigate the differential Galois group of a connection to get
some information on the Mordell-Weil lattices.
\begin{definition}
Let $(K,\;  ')\subset (L,\; ')$ be an extension of differential fields with field of constants $\CC$.
The differential Galois group $\DGal(L/K)$ is the group consisting of 
all the $K$-automorphisms $\sigma$ of $L$ such that $\sigma(f')=(\sigma(f))'$
for all $f\in L$.
\end{definition}

If $L$ is finitely generated as a $K$-algebra, say, by $p$ elements,
then $\DGal(L/K)$ can be embedded onto $GL(p,\CC)$, and
it is an algebraic group if considered as a subgroup of
$GL(p,\CC)$ in this embedding.

We apply this definition to $K=C(z)$, the derivation $'$ being
the differentiation with respect to some nonconstant function $z\in K$.
Given a connection $\nabla_\EEE $ on $\EEE$ over $\PP^1(\CC)$, we can consider a fundamental
matrix $\Phi$ of its solutions, and set $L$ to be the field
generated by all the matrix elements of $\Phi$. 
The group $\DGal(\nabla_\EEE )$ is defined to be $\DGal(L/K)$. See \cite{VDP}
for more details.
Note also that the $\DGal(\nabla_\EEE))$ of order $m$ if and only if the $\DGal(f^*(\nabla_\EEE))$ is
cyclic and order $2m$ if and only if there exists a point $(z_0,w_0)\in E$ of 
order $2m$. This description enables us to compute the
Mordell-Weil lattices of Beauville'semistable fibration.

 \renewcommand\refname{References}

\end{document}